\documentclass[a4pape,two side,11pt]{amsart}
\usepackage{lineno,hyperref,amsmath,amsfonts,amsthm}
\usepackage{amssymb,textcomp,epsfig,mathtools}
\usepackage[center]{caption}
\usepackage{xcolor,url}
\newtheorem*{Mthm}{Main Theorem}
\newtheorem{Thm}{Theorem}
\newtheorem{Prop}[Thm]{Proposition}
\newtheorem{Lem}[Thm]{Lemma}
\newtheorem{Cor}[Thm]{Corollary}

\newtheorem*{SunC}{Sun's 1-3-5 Conjecture}

\newtheorem{Rmk}{Remark}

\def\hh{{\mathcal H}}
\def\ll{{\mathcal L}}

\def\NN{\mathbb{N}}

\def\RR{\mathbb{R}}
\def\ZZ{\mathbb{Z}}

\def\Norm{\mathop{\rm N}\nolimits}
\def\dotp{\boldsymbol{\cdot}}
\usepackage{CJKutf8}
\newcommand{\hanyu}[1]{\begin{CJK}{UTF8}{gbsn}#1\end{CJK}}
\usepackage{orcidlink}
\begin{document}

\title{Zhi-Wei Sun's 1-3-5 Conjecture and Variations}

\author[A. Machiavelo\; \&\; N. Tsopanidis]{Ant\'onio
  Machiavelo\,\orcidlink{0000-0002-7595-7275}\;
  \&\; Nikolaos Tsopanidis}%
\address{Centro de Matem\'atica da Universidade do Porto\\ 4169-007
  Porto, Portugal}

\begin{abstract}
  In this paper, using quaternion arithmetic in the ring of Lipschitz
  integers, we present a proof of Zhi-Wei Sun's ``1-3-5 conjecture''
  for integral solutions, and reduce the natural solutions case to its
  verification up to $1.052 \times 10^{11}$. The computational
  verification was performed by the authors and a colleague,
  concluding the proof of Sun's 1-3-5 conjecture.  We also establish
  some variations of this conjecture.
\end{abstract}

\maketitle

\footnotesize


\bigskip \textbf{Warning:} This paper, published in the Journal of
Number Theory 222 (2021) 1–20.
doi\url{:https://doi.org/ 10.1016/j.jnt.2020.10.001} contains an
error that is corrected in an addendum at the end of this document.

\normalsize
\section{Introduction}

\thispagestyle{empty}

Lagrange’s four-square theorem states that any
$m\in\NN =\left\lbrace 0, 1, 2, \ldots \right\rbrace$ can be written
as the sum of four integer squares.  Zhi-Wei Sun, see \cite[Conjecture
4.3]{sun2017refining}, made the following conjecture.
\begin{SunC}\label{MainTheorem}
  Any $m\in\NN$ can be written as a sum of four squares,
  $m=x^2+y^2+z^2+t^2$ with $x, y, z, t\in\NN$, in such a way that
  $x+3y+5z$ is a perfect square.
\end{SunC}
Y.-C. Sun and Z.-W. Sun in \cite{sun2018variants} proved that any
$n\in\NN$ can be written as $x^2 + y^2 + z^2 + t^2$ with
$x, y, 5z, 5t \in \ZZ$ such that $x + 3y + 5z$ is a square
(cf. Theorem 1.8 of their paper). Moreover, H.-L. Wu and Z.-W. Sun in
\cite{WuSun2020} showed that any sufficiently large $n\in\NN$ with
$16 \nmid n$ can be written as $x^2 + y^2 + z^2 + t^2$ with
$x, y, z, t \in \ZZ$ such that $x + 3y + 5z\in\lbrace 1,4\rbrace$
(cf. Theorem 1.3(i) of their paper).

We present here a proof of that conjecture for all $m\in\NN$ with
$x, y, z, t\in\ZZ$, and a proof for all $m\not\equiv 0\pmod {16}$
greater than a specific constant, with $x, y, z, t\in\NN$. This,
together with computations done by the authors and Rog\'erio Reis in
\cite{report2020}, which checked the validity of the conjecture up to
that constant, completely proves the 1-3-5 conjecture. Moreover, we
establish some general results that correspond to variations of this
conjecture.

While the previous attempts to attack the conjecture used the theory
of quadratic forms, we use the arithmetic of the subring of Hamilton
quaternions known as the ring of Lipschitz integers,
$$\ll=\{a+bi+cj+dk\mid a,b,c,d\in\ZZ\},$$
where $i^2=j^2=k^2=ijk=-1$.  This ring is neither left nor right
Euclidean, as opposed to the ring of Hurwitz integers
$$\hh=\left\{a+bi+cj+dk\mid a,b,c,d\in\ZZ \, \text{ or } \,
  a,b,c,d\in\frac12+\ZZ \right\},$$
which is more commonly used because of that. Recall that the conjugate
of a quaternion $\alpha=a+bi+cj+dk$ is $$\bar{\alpha}=a-bi-cj-dk,$$
and its norm is $$\Norm(\alpha)=\alpha\bar{\alpha}=a^2+b^2+c^2+d^2.$$
We do not have unique factorization neither in $\ll$ nor in $\hh$, but
given a primitive (i.e.~not divisible by any natural number bigger
than one) Hurwitz integer $Q$, for any reordering of the primes of its
norm factorization into integral primes, $\Norm(Q)=p_0p_1\cdots p_s$,
there is a factorization of $Q$ into a product of Hurwitz primes
$Q = P_0P_1\cdots P_s$ such that
$\Norm(P_0) = p_0,\ldots, \Norm(P_s)=p_s$.  We say that the
factorization $P_0P_1\cdots P_s$ of $Q$ is modeled on the
factorization $p_0p_1\cdots p_s$ of $\Norm(Q)$. Moreover, the
factorization on a given model is unique up to unit-migration (see
\cite[Theorem 2, p.57]{conway2003quaternions}).

In what follows we will be dealing with Lipschitz integers, but in
this smaller ring there is also unique factorization modeled on
factorizations of the norm. In fact, Gordon Pall has proven in
\cite{pall1940arithmetic} that for a Lipschitz integer $v$ which is
primitive modulo $m$ (i.e.~the greatest common divisor of its
coefficients together with $m $ is $1$), where $m\mid N(v)$, $m$ is
odd and positive, there is a unique, up to left multiplication by
units, right divisor of $v$ of norm $m.$ This also holds for even $m$,
provided $v$ is actually primitive and $\frac{\Norm(v)}{m}$ is odd
(see Theorem $1$ and Corollary $1$ in \cite{pall1940arithmetic}).

For our purposes, uniqueness of factorization is not required. We only
need existence, which means that we may drop the condition for a
Lipschitz integer to be primitive, and we will still have a
factorization modeled on any factorization of its norm, including even
factors, because the only primes dividing $2$ in $\ll$ are, up to
associates, $1+i$, $1+j$, and $1+k$, and
$(1+i)(a+bi+cj+dk) = (a+bi-dj+ck) (1+i)$, with similar relations
holding for $1+j$ and $1+k$. Moreover, by Jacobi's $4$-square theorem,
factors of powers of $2$ are reduced to the factorization of $2$ and
$\pm 1\pm i\pm j\pm k$, which can easily be checked to be, up to
associates, products of two of the numbers $1+i$, $1+j$, or $1+k$.
\section{The general setting}

Let $a,b,c,d\in\ZZ$, and $m,n\in\NN$ be given.  Let us start by
describing conditions under which one can guarantee the existence of
$x,y,z,t\in\ZZ$ such that
\begin{equation}
  \label{eq:main}
  \left\{
    \begin{array}{rcl}
      x^2 + y^2 + z^2 + t^2  &=& m\\
      a x + b y + c z + d t &=& n^2 .
    \end{array}
  \right.
\end{equation}

Putting $\gamma = x+yi+zj+tk, \zeta = a+bi+cj+dk\in\ll$, these
equations are equivalent to
\begin{eqnarray}
  \label{eq:Nm}
  \Norm(\gamma)&=& m\\
  \label{eq:dotpn}
  \gamma\dotp\zeta = \Re(\bar{\gamma}\zeta) &=& n^2,
\end{eqnarray}
where the dot denotes here the usual inner product on $\RR^4$.  If one
sets $\delta = \bar{\gamma}\zeta$, it follows from \eqref{eq:dotpn}
that $\delta = n^2 + Ai+Bj+Ck$, for some $A,B,C\in\ZZ$, and
$m \Norm(\zeta) - n^4 = A^2+B^2+C^2$. By Legendre's three-square
theorem, see \cite[pp. 293--295]{legendre1808essai}, or for more recent
proofs see \cite{wojcik1971sums} and \cite{ankeny1957sums}, a
necessary condition for the solvability of \eqref{eq:main} is that one
has
\begin{equation}
  \label{eq:bound}
  n\leq \sqrt[4]{m\, \Norm(\zeta)},
\end{equation}
and that
\begin{equation}
  \label{eq:exceptions}
  m\,\Norm( \zeta) - n^4 \text{ is not of the form } 4^r(8s+7) \text{
  for any } r,s\in\NN.
\end{equation}

Assume now, conversely, that conditions \eqref{eq:bound} and
\eqref{eq:exceptions} are satisfied. Then, again by Legendre's
three-square theorem, there exist $A,B,C\in\ZZ$ such that
$m\, \Norm(\zeta) - n^4 = A^2+B^2+C^2$. Setting
$\delta = n^2+Ai+Bj+Ck$, one has $\Norm(\delta) = m\, \Norm(\zeta)$.
It then follows, by the existence of factorizations modeled on
factorizations of the norm in the ring of Lipschitz integers, that
there exists $\xi,\gamma\in\mathcal{L}$ such that
$\delta = \bar{\gamma}\xi$ and $\Norm(\xi)=\Norm(\zeta)$,
$\Norm(\gamma)=m$. It follows that $\gamma$ is a solution of
\begin{eqnarray}
  \label{eq:Nm2}
  \Norm(\gamma)&=& m\\
  \label{eq:dotpn2}
  \gamma\dotp\xi = \Re(\bar{\gamma}\xi) &=& n^2.
\end{eqnarray}
This proves the following.
\begin{Thm}
  \label{thm:generic}
  Let $m,n,\ell\in \NN$ be such that $n\leq \sqrt[4]{ m\ell}$, and
  assume that $m\ell-n^4$ is not of the form $4^r(8s+7)$ for any
  $r,s\in\NN$. Then, for some $a,b,c,d\in\NN$ such that
  $\Norm(a+bi+cj+dk)=\ell$, the system
  \begin{equation*}
    \left\{
      \begin{array}{rcl}
        m &=& x^2 + y^2 + z^2 + t^2\\
        n^2 &=& a x + b y + c z + d t.
      \end{array}
    \right.
  \end{equation*}
  has integer solutions.
\end{Thm}

\begin{proof}
  This follows from all that was written above, together with the fact
  that one can change the signs of $x,y,z,t$ so as to make $a,b,c,d$
  non-negative, if they are not already so.
\end{proof}

A direct consequence of Theorem~\ref{thm:generic} is the following.

\begin{Thm}
  Let $\zeta\in\ll$ and $m,n\in\NN$ be such that $\Norm(\zeta)m-n^4$
  is non-negative and not of the form $4^r(8s+7)$, for any
  $r,s\in\NN$. If $\zeta=a+bi+cj+dk\in\ll$, then the system
  \begin{equation*}
    \left\{
      \begin{array}{rcl}
        m &=& x^2 + y^2 + z^2 + t^2\\
        n^2 &=& a x + b y + c z + d t.
      \end{array}
    \right.
  \end{equation*}
  has integer solutions whenever
  \begin{equation*}\Norm(\zeta)=
    \left\{
      \begin{array}{rcl}
        & 1, 3, 5, 7, 11, 15, 23\\
        &2^g\\
        &3\cdot 2^g\\
        &7\cdot 2^g
      \end{array}
    \right.
  \end{equation*}
  where $g$ is odd and positive.
\end{Thm}
\begin{proof}
  Let $\ell\in\NN$, define the \textit{partition number} $P_4(\ell)$
  of $\ell$ into $4$ squares to be
  $$P_4(\ell)=\left|(a_1,a_2,a_3,a_4)\in\NN^4\,\mid\,a_1\geqslant
    a_2\geqslant a_3\geqslant a_4,
    \,\sum_{i=1}^{4}a_i^2=\ell\,\right|$$
  By Theorem~\ref{thm:generic}, it suffices to guarantee that
  $P_4(\Norm(\zeta))=1$, for all $\zeta\in\ll$, with $\Norm(\zeta)$
  running through all the values in the statement, which is true by
    \cite[Theorem 1]{lehmer1948partition}.
\end{proof}
\section{The 1-3-5 conjecture}

Let us now consider the existence of integer solutions for the system:
\begin{equation}
  \tag{1-3-5}\label{eq:135}
  \left\{
    \begin{array}{rcl}
      m &=& x^2 + y^2 + z^2 + t^2\\
      n^2 &=& x + 3 y + 5 z.
    \end{array}
  \right.
\end{equation}
Since the only possible Lipschitz integers of norm $35$, up to the
signs and the order of the coefficients, are $1+3i+5j$ and
$1+3i+3j+4k$, Theorem~\ref{thm:generic} immediately yields the
following result.
\begin{Prop}\label{1334or135}
  Let $n\leq\sqrt[4]{35\, m}$ be such that $35 m - n^4$ is not of the
  form $4^r(8s+7)$, for any $r,s\in\NN$. Then either the system
  \eqref{eq:135} has integer solutions, or the system
  \begin{equation}
    \tag{1-3-3-4}\label{eq:1334}
    \left\{
      \begin{array}{rcl}
        m &=& x^2 + y^2 + z^2 + t^2\\
        n^2 &=& x + 3 y + 3 z +4 t.
      \end{array}
    \right.
  \end{equation}
  has integer solutions.
\end{Prop}

Define $\mathcal{R} (P)$ to be the set of all Lipschitz integers
obtained from $P\in\ll$, by permuting and changing the signs of its
coordinates. For $\alpha,\alpha'\in\ll$, we say that $\alpha'$ is in
the same \emph{decomposition class} as $\alpha$, and write
$ \alpha'\sim \alpha$, if
$\mathcal{R} (\alpha')=\mathcal{R} (\alpha)$.

\textbf{From now on}, we set $\alpha = 1+3i+5j$ and
$\beta=1+3i+3j+4k$.  In sections $4,5,6$ and $7$ we will prove that
the system \eqref{eq:135} always has a solution for all $m\in\NN$ with
$x,y,z,t\in\ZZ$. The natural solution case will be handled in the last
section. The biggest part of this paper will be focused on proving the
following theorem.
\begin{Thm}\label{Thm135int}
  Let $m,n\in\NN$ be such that $35m-n^4$ is non-negative and not of
  the form $4^r(8s+7)$, for any $r,s\in\NN$. Then
  \begin{itemize}
  \item[i)]If $m\equiv 0\pmod 3$, then the system \eqref{eq:135} has
    integer solutions whenever $n\not\equiv 0\pmod 3$ and
    $n\not\equiv 0 \pmod 5$, i.e.~$(n,15)=1$.
  \item[ii)] If $m\equiv 1 \pmod 3$, then the system \eqref{eq:135}
    has integer solutions whenever $n\equiv 0\pmod 3$ and
    $n\not\equiv 0 \pmod 5$.
  \item[iii)]If $m\equiv -1\pmod 3$, then the system \eqref{eq:135}
    has integer solutions whenever $n\not\equiv 0\pmod 3$,
    $n\not\equiv 0 \pmod 5$ and $n\not\equiv 0 \pmod 7$,
    i.e.~$(n,105)=1$.
  \end{itemize}
\end{Thm}
Since the condition ``$35m-n^4$ is not of the form $4^r(8s+7)$, for
any $r,s\in\NN$'' holds often enough, this theorem shows more than
what the integer case of the 1-3-5 conjecture asserts.
As it is suggested from the statement of the theorem, we need to work
modulo $3,5$ and $7$.

Let us now establish the framework within which we are going to
work. We assume that $m,n\in\NN$ are such that $35 m - n^4$ is
non-negative and not of the form $4^r(8s+7)$, for any $r,s\in\NN$. Like
in the first section, this implies that there exist $A,B,C\in\NN$
such that $35 m - n^4=A^2+B^2+C^2$. Letting
$\delta=n^2+Ai+Bj+Ck\in\ll$, we have that $\Norm(\delta)=35m$, and
therefore there exist $\zeta,\gamma\in\ll$ with $\Norm(\zeta)=35$ and
$\Norm(\gamma)=m$, such that $\delta=\gamma\zeta$. Then
$\Norm(\zeta)=35$, and so $\zeta\sim\beta$ or $\zeta\sim\alpha$.

If $\zeta\in\mathcal{R}(\alpha)$ then the system \eqref{eq:135} has
integer solutions and we are done.  If $\zeta\in\mathcal{R}(\beta)$,
then there exist a $\gamma'$, obtained from appropriate sign and
coefficient changes of $\gamma$, with
$\Norm(\gamma')=\Norm(\gamma)=m$, such that
$\Re(\gamma'\beta)=\Re(\gamma\zeta)=n^2$. Therefore, we may assume,
without loss of generality, that $\zeta=\beta$. Let
$\gamma=x-yi-zj-tk$. Performing the multiplication $\gamma\beta$
yields
$$\delta=(x+3y+3z+4t)+(3x-y-4z+3t)i+(3x+4y-z-3t)j+(4x-3y+3z-t)k,$$
so we have, for future reference:
\begin{equation}\label{eq:ABC2}
  \left\{
    \begin{array}{rcl}
      n^2 &=& x + 3 y + 3 z +4 t\\
      A &=& 3 x - y - 4 z + 3 t\\
      B &=& 3 x + 4 y - z - 3 t\\
      C &=& 4 x - 3 y + 3 z - t.
    \end{array}
  \right.
\end{equation}

We now point out the main idea behind what is going to be done in the
next sections:
\begin{Rmk}
  \label{Remark}
  For any $\rho\in\ll\setminus\{0\}$, one has
  $\Re(\rho^{-1}\delta\rho) = \Re(\delta)$ and
  $\Norm(\rho^{-1}\gamma\rho)=\Norm(\gamma)$. Since, for any
  $\sigma\in\ll\setminus\{0\}$,
  \begin{equation*}
    \rho^{-1}\delta\rho = \rho^{-1}\gamma\beta\rho =
    \rho^{-1}\gamma\sigma\,\sigma^{-1}\beta\rho, 
  \end{equation*}
  we see that if one can find $\rho,\sigma\in\ll\setminus\{0\}$ such
  that $\sigma^{-1}\beta\rho=\alpha'$ and
  $\rho^{-1}\gamma\sigma\in\ll$, with $\alpha'\in\mathcal{R}(\alpha)$
  and $\Norm(\rho)=\Norm(\sigma)$, then from a solution
  $(x,y,z,t)\in\ZZ^4$ for \eqref{eq:1334} one can obtain a solution in
  $\ZZ^4$ for \eqref{eq:135}.
\end{Rmk}

We will be using this in the case where
$\Norm(\rho) = \Norm(\sigma) = p$, an odd prime, and in order to apply
this remark, we will need conditions on $\gamma$ that guarantee
$\rho^{-1}\gamma\sigma\in\ll$, which is the same as
$\bar{\rho}\gamma\sigma\equiv 0\pmod{p}$. Those conditions can be
obtained by using Corollary~8 in \cite{pall1940arithmetic}, which will
be here applied in the following way. Since $\bar{\rho} \gamma\sigma$
and $\bar{\rho}\sigma$ have the same right and left divisors of norm
$p$, \textbf{when $\bar{\rho}\sigma$ is primitive modulo $p$}, Pall's
result implies that there is a $k_\gamma\in\ZZ$ such that
$\bar{\rho} \gamma\sigma \equiv k_\gamma \bar{\rho} \sigma \pmod{p}$.
But then, taking the conjugate congruence, adding both, and using the
fact that $\Re(rs) = \Re(sr)$ for all $r,s\in\ll$, one gets
$$ \gamma\cdot \rho\bar{\sigma} \equiv k_\gamma\,
\rho\cdot \sigma \pmod{p}.$$
When $\rho\cdot\sigma\not\equiv 0 \pmod{p}$, which is the same as
$p\nmid \Re(\bar{\rho}\sigma)$, then one has
$k_\gamma \equiv\frac{\gamma\cdot \rho\bar{\sigma}}{\rho\cdot\sigma}
\pmod{p}$, and one concludes that
\begin{equation}
  \label{eq:Riesz1}
  \bar{\rho} \gamma\sigma \equiv \frac{1}{\rho\cdot\sigma}\,
  (\gamma\cdot\rho\bar{\sigma})\, \bar{\rho} \sigma \pmod{p}, \text{
    for all } \gamma\in\ll.
\end{equation}

If $\rho\cdot\sigma\equiv 0 \pmod{p}$, then
$\gamma\cdot \rho\bar{\sigma} \equiv 0 \pmod{p}$ for all
$\gamma\in\ll$, and in particular
$\rho \bar{\sigma}\equiv 0 \pmod{p}$. Hence $\sigma = u \rho$, for
some $u\in\ll^*$. Before dealing with this possibility, consider the
case \textbf{when $\bar{\rho}\sigma$ is not primitive modulo $p$}. This
means that $\sigma$ is a right associate of $\rho$, and so we can
assume, without loss of generality, that $\sigma=\rho$. Using
coordinates, we can explicitly see that, also in this case,
$\bar{\rho} \gamma\rho$ has proportional coordinates modulo $p$, and
thus, as above, there are $\varepsilon,\delta\in\ll$ such that
$\bar{\rho} \gamma\rho \equiv (\gamma\cdot\varepsilon)\, \delta
\pmod{p}$, for all $\gamma\in\ll$. If $\rho=a+bi+cj+dk$, with
$c^2+d^2\neq 0$, it can be seen that one can take
$\delta=(a^2+b^2)i+(bc-ad)j+(ac+bd)k$, and $\varepsilon$ can then be
easily computed for any given $\rho$. Finally, if $\rho=a+bi$, with
$b\neq 0$, one can take $\delta=2aj-2bk$ and $\varepsilon = aj+bk$.

Finally, for the case $\sigma = u \rho$, with $u\in\ll^*$, one applies
what we just saw to $\gamma u$, to obtain $\varepsilon, \delta$ such
that
$\bar{\rho} \gamma \sigma = \bar{\rho} \gamma u\rho \equiv (\gamma
u\cdot\varepsilon)\, \delta \equiv (\gamma\cdot\varepsilon\bar{u})\,
\delta \pmod{p}$.

Thus, the following holds:
\begin{Prop}
  \label{prop:criterion}
  Given $\rho, \sigma\in\ll$ with norm $p$, then there are
  $\varepsilon,\delta\in\ll$ such that
  $\bar{\rho} \gamma\sigma \equiv (\gamma\cdot\varepsilon)\, \delta
  \pmod{p}$, for all $\gamma\in\ll$. Moreover, for any $\rho, \sigma$,
  one can easily compute $\varepsilon$, which then yields the
  following criterion:
  $$\rho^{-1} \gamma\sigma\in\ll \iff \gamma\cdot
  \varepsilon \equiv 0\pmod{p}.$$
\end{Prop}

\textbf{From now on}, we assume that $(x_0,y_0,z_0,t_0)\in\ZZ^4$ is a
solution of the system \eqref{eq:1334}, we set
$\gamma_0=x_0-y_0i-z_0j-t_0k$, and we are going to show that from this
solution one can construct a solution for the \eqref{eq:135} system,
by using Remark~\ref{Remark} and Proposition~\ref{prop:criterion}.
\section{Using primes in $\ll$ with norm 3}

Let $\rho = 1+i-j.$ One can easily check that
\begin{equation}
  \label{P3-0}
  \beta\rho = \sigma\alpha',
\end{equation}
where $\sigma=1+i+j$, $\alpha'=5+3i+j\in\mathcal{R}(\alpha)$, and that
\begin{equation*}
  \rho^{-1} \gamma_0 \, \sigma\in\ll \iff x_0-z_0-t_0\equiv 0 \pmod{3}.
\end{equation*}
Thus, since $(x_0,y_0,z_0,t_0)\in\ZZ^4$ is a solution of
\eqref{eq:1334}, it follows by Remark~\ref{Remark} that when this
congruence holds, $\rho^{-1} \gamma_0 \, \sigma$ yields a
solution of \eqref{eq:135}.

Now, there are $4 $ right non-associated primes above $3$, and for the
ones other than $\rho_1=1+i-j$, multiplying by $\beta$ on the left
yields:
\begin{eqnarray}
  \label{P3-1}\beta(1-i-j)&=&(1+j+k)(3-5j+k)\\
  \label{P3-2}\beta(1+i+j)&=&(1-j+k)(-3+4i+j+3k)\\
  \label{P3-3}\beta(1-i+j)&=&(1+i+k)(3-i+4j+3k).
\end{eqnarray}

Using \eqref{P3-1} instead of \eqref{P3-0}, and repeating the same
argument, we get that, if $x_0-y_0-t_0\equiv 0 \pmod{3}$, then the
system \eqref{eq:135} has an integer solution; using \eqref{P3-2} for
$x_0+y_0-t_0\equiv 0 \pmod{3}$, one obtains yet another integer
solution for the system \eqref{eq:1334}; and using \eqref{P3-3} for
$x_0+z_0-t_0\equiv 0 \pmod{3}$, one gets again another integer
solution for the system \eqref{eq:1334}. In the last two cases we
obtain no direct information for the solvability of the system
\eqref{eq:135}, but the extra solutions we get, using \eqref{P3-2} and
\eqref{P3-3}, for the system \eqref{eq:1334} are going to prove
instrumental for our proof.  Later on we will need to write these
extra solutions in terms of $x_0,y_0,z_0,t_0$. For now, we note that
the above discussion has proved the following.
\begin{Prop}\label{mod3Conditions}
  Let $m,n\in\NN$ be such that $35m-n^4$ is non-negative and not of the
  form $4^r(8s+7)$, for any $r,s\in\NN$. For a solution
  $(x_0,y_0,z_0,t_0)\in\ZZ^4$ of the system \eqref{eq:1334}, if either
  of the following holds:
  \begin{itemize}
  \item[i)] $x_0-y_0-t_0\equiv 0\pmod 3$, or
  \item[ii)]$x_0-z_0-t_0\equiv 0\pmod 3$,
  \end{itemize}
  then the system \eqref{eq:135} has an integer solution.
\end{Prop}
\section{Using primes in $\ll$ with norm 5}

Much like as we did in the previous section, where we used the primes
above $3$ to see that a solution $(x_0,y_0,z_0,t_0)\in\ZZ^4$ for the
system \eqref{eq:1334} either yields conditions for the solvability of
the system \eqref{eq:135}, or another solution of the system
\eqref{eq:1334}, here we will use the primes above $5$ to do something
analogous, and we will actually calculate the new solutions for the
system \eqref{eq:1334}, since we will need to use those explicit
expressions.

Taking representatives of all the six primes of norm $5$, up to right
associates, and multiplying by $\beta$ on the left, we get
\begin{eqnarray*}
  \beta(1+2i)&=&(j-2k)(3-4i+3j-k)\\
  \beta(1+2j)&=&(2+i)(-3-i+4j+3k)\\
  \beta(1-2k)&=&(1-2i)(3+3i+j+4k)\\
  \beta(1-2i)&=&(2+i)(3-i+5k)\\
  \beta(1-2j)&=&(-1+2i)(3-5i-j)\\
  \beta(1+2k)&=&(j-2k)(-3+5j-k).
\end{eqnarray*}
For $\delta=\gamma_0\beta$, we then see that
\begin{equation}
  \label{eq:primesover5}
  \begin{split}
    (1+2i)^{-1}\delta (1+2i)
    &=[(1+2i)^{-1}\,\gamma_0\,(j-2k)](3-4i+3j-k)\\
    (1+2j)^{-1}\delta (1+2j)
    &=[(1+2j)^{-1}\,\gamma_0\,(2+i)](-3-i+4j+3k)\\
    (1-2k)^{-1}\delta (1-2k)
    &=[(1-2k)^{-1}\,\gamma_0\,(1-2i)](3+3i+j+4k)\\
    (1-2i)^{-1}\delta (1-2i)
    &=[(1-2i)^{-1}\,\gamma_0\,(2+i)](3-i+5k)\\
    (1-2j)^{-1}\delta (1-2j)
    &=[(1-2j)^{-1}\,\gamma_0\,(-1+2i)](3-5i-j)\\
    (1+2k)^{-1}\delta (1+2k)
    &=[(1+2k)^{-1}\,\gamma_0\,(j-2k)](-3+5j-k).
  \end{split}
\end{equation}
Denoting the expressions in the brackets by $\gamma_i$,
$i=1,\dots, 6$, respectively,
%
%
one sees that if any of $\gamma_4,\gamma_5,\gamma_6$ is in $\ll$, then
the system \eqref{eq:135} would have integer solutions by
Remark~\ref{Remark}, and we are done. One has, using \eqref{eq:ABC2}
and Proposition~\ref{prop:criterion},
\begin{equation*}
    \begin{array}{l}
      \gamma_4\in\ll \iff t_0\equiv 3z_0\pmod 5\iff n^2\equiv 2A\pmod
      5\\  
      \gamma_5\in\ll \iff x_0 - 2y_0 + 2z_0 + t_0\equiv 0\pmod 5
      \iff A\equiv 0\pmod 5\\ 
      \gamma_6\in\ll \iff x_0 - 2y_0 + z_0 - 2t_0\equiv 0\pmod5
      \iff n^2\equiv-A\pmod 5. 
    \end{array}
\end{equation*}
Therefore, we just proved the following.
\begin{Prop}\label{mod5Conditions}
  Let $m,n\in\NN$ be such that $35m-n^4$ is non-negative and not of
  the form $4^r(8s+7)$, for any $r,s\in\NN$. If
  $(x_0,y_0,z_0,t_0)\in\ZZ^4$ is a solution of the system
  \eqref{eq:1334}, and $A = 3 x_0 - y_0 - 4 z_0 + 3 t_0$ satisfies any
  one of the following congruences:
  \begin{itemize}
  \item[(i)] $A\equiv 0\pmod 5, $
  \item[(ii)]$n^2\equiv 2A\pmod 5,$
  \item[(iii)]$n^2\equiv -A\pmod 5,$
  \end{itemize}
  then the system \eqref{eq:135} has an integer solution.
\end{Prop}

We notice that if $(x_0,y_0,z_0,t_0)$ is a solution of the system
\eqref{eq:1334}, then $(x_0,z_0,y_0,t_0)$ is a solution of it as well.
Therefore we also have:
\begin{Cor}\label{mod5ConditionsCor}
  Let $m,n\in\NN$ be such that $35m-n^4$ is non-negative and not of the
  form $4^r(8s+7)$, for any $r,s\in\NN$. If $(x_0,y_0,z_0,t_0)\in\ZZ^4$
  is a solution of the system \eqref{eq:1334} such that any of the
  following congruences hold:
  \begin{itemize}
  \item[(i)] $t_0\equiv 3y_0\pmod 5$,
  \item[(ii)]$x_0 + 2y_0 - 2z_0 + t_0\equiv 0\pmod 5$,
  \item[(iii)]$x_0 +y_0 - 2z_0 - 2t_0\equiv 0\pmod 5$,
  \end{itemize}
  then the system \eqref{eq:135} has an integer solution.
\end{Cor}
Let us look at $\gamma_1,\gamma_2,\gamma_3$ now. Using once more
\eqref{eq:ABC2} and Proposition~\ref{prop:criterion}, one gets:
\begin{equation*}
  \begin{array}{l}
    \gamma_1\in\ll \iff y_0\equiv 3x_0\pmod 5 \iff
    n^2\equiv -2A \pmod 5\\ 
    \gamma_2\in\ll \iff x_0 - 2y_0 - 2z_0 -t_0 \equiv 0\pmod 5
    \iff n^2\equiv 0\pmod 5\\ 
    \gamma_3\in\ll \iff x_0 - 2y_0 - z_0 + 2t_0\equiv 0\pmod 5
    \iff n^2\equiv A\pmod 5. 
  \end{array}
\end{equation*}
Note that for $n^2\not\equiv 0\pmod 5$, either
$n^2\equiv\pm A\pmod 5$, $n^2\equiv\pm 2A\pmod 5$ or
$A\equiv 0\pmod 5$.  We have seen what happens if
$n^2\equiv- A\pmod 5$, $n^2\equiv 2A\pmod 5$ and $A\equiv 0\pmod 5$,
hence we just need to see what happens on the other two remaining
cases:
\begin{itemize}
\item If $n^2\equiv A\pmod 5$, then
  $x_0 - 2y_0 - z_0 + 2t_0\equiv 0\pmod 5$, so $\gamma_3\in\ll$.
  Since
  \begin{center}
    \small
    $\gamma_3=\frac{x_0 - 2y_0 + 4z_0 + 2t_0}{5} - \frac{2x_0 + y_0 -
      2z_0 + 4t_0}{5}i-\frac{4x_0 + 2y_0 + z_0 - 2t_0}{5}j+\frac{2x_0
      - 4y_0 - 2z_0 - t_0}{5}k$,
  \end{center}
  and, according to \eqref{eq:primesover5}, $\gamma_3$ yields the
  element $\beta^* = 3+3i+j+4k\in\mathcal{R}(\beta)$, it follows that
  \begin{center}
    \small
    $\gamma^*=\frac{4x_0 + 2y_0 + z_0 - 2t_0}{5} - \frac{2x_0 + y_0 -
      2z_0 + 4t_0}{5}i - \frac{x_0 - 2y_0 + 4z_0 + 2t_0}{5}j
    +\frac{2x_0 - 4y_0 - 2z_0 - t_0}{5}k$
  \end{center}
  satisfies
  $\Re(\gamma^*\beta)=\Re(\gamma_3\beta^*)=\Re(\gamma_0\beta)$, and
  thus $\gamma^*$ yields a solution of \eqref{eq:1334}.

  If we denote the coordinates of the conjugate of $\gamma^*$ by $x_1,y_1,z_1,t_1$,
  using the fact that $ x_0 - 2y_0 - z_0 + 2t_0 = 5\kappa,$ for some
  $\kappa\in\ZZ$, we have
  \begin{equation}\label{n^2=Amod5}
    \left\{
      \begin{array}{rcl}
        x_1&=&x_0-\kappa\\
        y_1&=&y_0+2\kappa\\
        z_1&=&z_0+\kappa\\
        t_1&=&t_0-2\kappa.
      \end{array}
    \right.
  \end{equation}
\item If $n^2\equiv -2A\pmod 5$, then $y_0\equiv 3x_0\pmod 5$, and so
  $\gamma_1\in\ll$. One then sees, as above, that
  $$\Re\left[\left(\frac{-4x_0+3y_0}{5}- \frac{3x_0+4y_0}{5}i-z_0j-
      t_0k\right)\beta\right] = \Re(\gamma\beta),$$
  and therefore
  $$\left(x_2,y_2,z_2,t_2\right) =
  \left(\frac{-4x_0+3y_0}{5},\frac{3x_0+4y_0}{5},z_0,t_0\right)$$ is
  another integer solution of the system \eqref{eq:1334} obtained from
  the solution $(x_0,y_0,z_0,t_0)\in\ZZ^4$.  We have
  $y_0=3x_0+5\lambda$, for some $\lambda\in\ZZ$, and thus
  \begin{equation}\label{n^2=-2Amod5}
    \left\{
      \begin{array}{rcl}
        x_2&=&x_0+3\lambda\\
        y_2&=&y_0-\lambda\\
        z_2&=&z_0\\
        t_2&=&t_0.
      \end{array}
    \right.
  \end{equation}
  is another integer solution of \eqref{eq:1334}.
\end{itemize}
Now we are ready to prove the following:
\begin{Prop}\label{135int 1+2}
  Let $m,n\in\NN$ be such that $35m-n^4$ is non-negative and not of the
  form $4^r(8s+7)$, for any $r,s\in\NN$. The following holds:
  \begin{itemize}
  \item[i)]If $m\equiv 0\pmod 3$, then the system \eqref{eq:135} has
    integer solutions for all $n\in\NN$ with $(n,15)=1$.
  \item[ii)] If $m\equiv 1 \pmod 3$, then the system \eqref{eq:135}
    has integer solutions for all $n\equiv 0\pmod 3$ with
    $5\nmid n$.
  \end{itemize}
\end{Prop}

\begin{proof}
  As above, we may assume the existence of a solution
  $(x_0,y_0,z_0,t_0)\in\ZZ^4$ of the system \eqref{eq:1334}.  Note
  that if $m\equiv 0\pmod 3$ and $n^2\equiv 1 \pmod 3$, or if
  $m\equiv 1\pmod 3$ and $n^2\equiv 0 \pmod 3$, then
  $35m-n^4\equiv -1\pmod 3$. Therefore $A^2+B^2+C^2\equiv -1\pmod 3$,
  and since the squares modulo $3$ are $0$ and $1$, we have that
  exactly one of the $A,B,C$ is $0$ modulo $3$, and the other two are
  $\pm1$ modulo $3$.  From \eqref{eq:ABC2} and for a solution
  $(x_0,y_0,z_0,t_0)\in\ZZ^4$ of \eqref{eq:1334}, we see that
  \begin{equation*}
    \left\{
    \begin{array}{rcl}
      n^2 &\equiv& x_0 +  t_0 \pmod 3\\
      A   &\equiv&  - y_0 - z_0 \pmod 3\\
      B   &\equiv&  y_0 - z_0 \pmod 3 \\
      C   &\equiv&  x_0 - t_0 \pmod 3.
    \end{array}
  \right.
  \end{equation*}
  We now consider all possibilities for the congruence classes of
  $A,B,C$ modulo 3. In each one of the following cases, one sees that
  one can use Proposition~\ref{mod3Conditions} to show that the system
  \eqref{eq:135} has integer solutions:
  \begin{itemize}
  \item If $A\equiv0\pmod 3$ and $B\equiv C\pmod 3$, then it is easy
    to see that $x_0+2z_0+2t_0\equiv0\pmod 3$.
  \item If $A\equiv0\pmod 3$ and $B\equiv -C\pmod 3$, then
    $x_0+2y_0+2t_0\equiv0\pmod 3$.
  \item If $C\equiv0\pmod 3$ and $A\equiv B\pmod 3$, then
    $x_0+2y_0+2t_0\equiv0\pmod 3$.
  \item If $C\equiv0\pmod 3$ and $A\equiv -B\pmod 3$, then
    $x_0+2z_0+2t_0\equiv0\pmod 3$.
  \item If $B\equiv0\pmod 3$ and $A\equiv C\pmod 3$, then
    $x_0+2y_0+2t_0\equiv0\pmod 3$.
  \end{itemize}
  There is only one remaining case:
  \begin{itemize}
  \item If $B\equiv0\pmod 3$ and $A\equiv -C\pmod 3$, then we have
    that $x_0+y_0+2t_0\equiv x_0+z_0+2t_0\equiv0\pmod 3$.
    Proposition~\ref{mod3Conditions} does not yield the claim this
    time. Instead, we are going to use the results from the previous
    section. For $n^2\not\equiv0\pmod 5$, we have the following cases:
    \begin{itemize}
    \item If we have that $A\equiv 0\pmod 5$ or $n^2\equiv 2A\pmod 5$ or
      $n^2\equiv -A\pmod 5$,  by
      Proposition~\ref{mod5Conditions}, the system \eqref{eq:135} has
      integer solutions.
    \item If $n^2\equiv A\pmod 5$ then the solution \eqref{n^2=Amod5}
      of the system \eqref{eq:1334} satisfies
      $x_1+2y_1+2t_1\equiv 2(x_0+z_0+2t_0)\equiv 0\pmod 3$. Therefore,
      Proposition~\ref{mod3Conditions} yields the claim.
    \item If $n^2\equiv -2A\pmod 5$, then the solution
      \eqref{n^2=-2Amod5} of the system \eqref{eq:1334} satisfies
      $x_2+2y_2+2t_2\equiv x_0+y_0+2t_0\equiv 0\pmod 3$. Therefore,
      Proposition~\ref{mod3Conditions} yields the claim again.
    \end{itemize}
  \end{itemize}
\end{proof}

The case $m\equiv -1 \pmod 3$ of \eqref{Thm135int} is the only one
left to be treated. For that case, similarly to the above, we can show
the following:
\begin{Prop}\label{mod3bad case}
  Let $m,n\in\NN$, $m\equiv -1\pmod 3$, $n\not\equiv 0\pmod 3$, be
  such that $35m-n^4$ is non-negative and not of the form $4^r(8s+7)$, for
  any $r,s\in\NN$. Then, either the system \eqref{eq:135} has integer
  solutions, or if $(x_0,y_0,z_0,t_0)\in\ZZ^4$ is a solution of the
  system \eqref{eq:1334}, we must have either
  $x_0+y_0+2t_0\equiv 0\pmod 3$ and $z_0\equiv 0\pmod 3$, or
  $ x_0+z_0+2t_0\equiv 0\pmod 3$ and $y_0\equiv 0\pmod 3$.
\end{Prop}

\begin{proof}
  Let $m\equiv -1\pmod 3$ and $n\not\equiv 0\pmod 3$, which means that
  $n^2\equiv 1 \pmod 3$, then $35m-n^4\equiv 0 \pmod 3$, so that
  $A^2+B^2+C^2\equiv 0\pmod 3$.  Therefore $A^2\equiv B^2\equiv C^2\not\equiv 2\pmod 3$, and
  \begin{itemize}
  \item If $A\equiv B\equiv C\pmod 3$, then
    $x_0+z_0+2t_0\equiv 0\pmod 3$ and $y_0\equiv 0\pmod 3$.
  \item If $A\equiv-B\equiv-C\pmod 3$, then
    $x_0+2y_0+2t_0\equiv 0\pmod 3$, and
    Proposition~\ref{mod3Conditions} applies.
  \item If $A\equiv B\equiv -C\pmod 3$, then
    $x_0+2z_0+2t_0\equiv 0\pmod 3$, and again
    Proposition~\ref{mod3Conditions} applies.
  \item If $A\equiv -B\equiv C\pmod 3$, then
    $x_0+y_0+2t_0\equiv 0\pmod 3$ and $z_0\equiv 0\pmod 3$.
  \end{itemize}
\end{proof}

In order to complete the proof of the case $m\equiv-1\pmod 3$ of
Theorem \ref{Thm135int}, we need to work modulo $7$ as well, since
the above methods are not enough to cover every possibility.
\section{Using primes in $\ll$ with norm 7}

Let $\rho_1= 1+i+j+2k$, $\rho_2= 1-i-j-2k$, $\rho_3= 1-i+j-2k$,
$\rho_4= 1+i-j+2k$, $\rho_5=1+i+j-2k$, $\rho_6= 1-i-j+2k$,
$\rho_7= 1+i-j-2k$, and , $\rho_8= 1-i+j+2k$ be representatives of all
the $8$ right non-associate primes of norm $7$. Multiplying them all
by $\beta $ on the left, as we did before for the primes of norm $3$
or $5$, we get:
\begin{equation*}
  \begin{array}{rcl}
    \beta\rho_1 &=& (1-i+2j-k)(-3 - 3i + 4j + k)\\
    \beta\rho_2 &=& (1 -i+2j-k)( 3 + i - 4j + 3k)\\
    \beta\rho_3 &=& ( 2 -i-j+k)( 4 + i + 3j + 3k)\\
    \beta\rho_4 &=& (1+i-j-2k)(1 + 3i + 3j - 4k)\\
    \beta\rho_5 &=& (-2+i+j-k)(-i - 5j - 3k)\\
    \beta\rho_6 &=& (-1 -2i -j -k)(-3  + j - 5k)\\        
    \beta\rho_7 &=& (-2 +i+j-k)(-3i -5j + k)\\
    \beta\rho_8 &=& (1 + i - j - 2k)(-3 + 5i + j).
  \end{array}
\end{equation*} 

Denoting by $\sigma_i$ the corresponding prime above $7$ that shows up
on the right side, and setting
$\hat{\gamma}_i = \rho_i^{-1} \gamma_0\sigma_i$, one has
\begin{equation}
  \label{eq:primesover7}
  \begin{split}
    \rho_1^{-1}\gamma_0\beta \rho_1 &= \hat{\gamma}_1\, (-3-3i+4j+k)\\
    \rho_2^{-1}\gamma_0\beta \rho_2 &= \hat{\gamma}_2\, (3+i-4j+3k)\\
    \rho_3^{-1}\gamma_0\beta \rho_3 &= \hat{\gamma}_3\, (4+ i+3j+3k)\\
    \rho_4^{-1}\gamma_0\beta \rho_4 &= \hat{\gamma}_4\, (1+3i+3j-4k)\\
    \rho_5^{-1}\gamma_0\beta \rho_5 &= \hat{\gamma}_5\, (-i-5j-3k)\\
    \rho_6^{-1}\gamma_0\beta \rho_6 &= \hat{\gamma}_6\, (-3 +j-5k)\\
    \rho_7^{-1}\gamma_0\beta \rho_7 &= \hat{\gamma}_7\, (-3i-5j+k)\\
    \rho_8^{-1}\gamma_0\beta \rho_8 &= \hat{\gamma}_8\, (-3+5i+j).
  \end{split}
\end{equation}

If any of the $\hat{\gamma}_i$ for $i=5,6,7,8$ is in $\ll$, then the
system \eqref{eq:135} would have integer solutions, and we are
done. Using \eqref{eq:ABC2} and Proposition~\ref{prop:criterion}, one
deduces
\begin{equation*}
  \begin{array}{l}
    \hat{\gamma}_5\in\ll \iff x_0+2y_0+z_0+t_0\equiv 0\pmod 7
    \iff A\equiv 0\pmod 7\\ 
    \hat{\gamma}_6\in\ll
    \iff x_0 - 2y_0 + 3t_0\equiv 0\pmod 7 \iff n^2\equiv A\pmod 7\\ 
    \hat{\gamma}_7\in\ll
    \iff y_0+2z_0+3t_0\equiv 0\pmod 7\iff n^2\equiv-2A\pmod 7 \\ 
    \hat{\gamma}_8\in\ll \iff x_0 + y_0 - z_0 - 2t_0\equiv 0\pmod
    7\iff n^2\equiv-4A\pmod 7 . 
  \end{array}
\end{equation*}
Therefore, we have proved the following.
\begin{Prop}\label{mod7Conditions}
  Let $m,n\in\NN$ be such that $35m-n^4$ is non-negative and not of the
  form $4^r(8s+7)$, for any $r,s\in\NN$. If
  $(x_0,y_0,z_0,t_0)\in\ZZ^4$ is a solution of the system
  \eqref{eq:1334}, and if any of the following holds:
  \begin{itemize}
  \item[(i)] $A\equiv 0\pmod 7$
  \item[(ii)]$n^2\equiv A\pmod 7$
  \item[(iii)]$n^2\equiv -2A\pmod 7$
  \item[(iii)]$n^2\equiv -4A\pmod 7$
  \end{itemize}
  then the system \eqref{eq:135} has an integer solution.
\end{Prop}

Now, let us look at $\hat{\gamma}_i$, for $i=1,2,3,4$. Applying once
more \eqref{eq:ABC2} and Proposition~\ref{prop:criterion}, one has:
\begin{equation*}
  \begin{array}{l}
    \hat{\gamma}_1\in\ll \iff x_0+4z_0+2t_0\equiv 0\pmod 7 \iff
    n^2\equiv 4A \pmod 7\\ 
    \hat{\gamma}_2\in\ll\iff x_0 - y_0 + 2z_0 -t_0 \equiv 0\pmod 7
    \iff n^2\equiv 2A\pmod 7\\ 
    \hat{\gamma}_3\in\ll \iff x_0 + 4y_0 - 2z_0 \equiv 0\pmod 7
    \iff n^2\equiv -A\pmod 7\\ 
    \hat{\gamma}_4\in\ll \iff x_0 + 3y_0 + 3z_0 +4t_0 \equiv 0\pmod 7
    \iff n^2\equiv 0\pmod 7. 
  \end{array}
\end{equation*}  
If any of the $\hat{\gamma}_i$ for $i=1,2,3,4$ is in $\ll$, then we
will have another solution for the system $\eqref{eq:1334}$. We do not
care for $\hat{\gamma}_4$, as the statement of Theorem~\ref{Thm135int}
suggests, and we will examine each of the cases
$\hat{\gamma}_1,\hat{\gamma}_2,\hat{\gamma}_3\in\ll$ separately. Note
that if $n\not \equiv 0 \pmod 7$, then we have either
$n^2\equiv\pm A\pmod 7$, $n^2\equiv\pm 2A\pmod 7$,
$n^2\equiv\pm 4A\pmod 7$, or $A\equiv0\pmod 7$.
\begin{itemize}
\item If $\hat{\gamma}_1\in\ll$, then $n^2\equiv 4A\pmod 7$ and
  $x_0+4z_0+2t_0\equiv 0\pmod 7$, which means that
  $x_0+4z_0+2t_0=7\mu$, for some $\mu\in\ZZ$. Looking at the
  coordinates of $\hat{\gamma}_1$, rearranging them and changing signs
  accordingly, one sees that for
  $$\gamma_1^*=\frac{6x_0 + 3z_0 - 2t_0}{7} - \frac{3x_0 - 2z_0 +
    6t_0}{7}i - y_0j + \frac{2x_0 - 6z_0 - 3t_0}{7}k$$ one has
  $\Re(\gamma_1^*\beta)=\Re(\gamma\beta)$, and hence
  $$(\hat{x}_1,\hat{y}_1,\hat{z}_1,\hat{t}_1)=\left(\frac{6x_0 + 3z_0
      - 2t_0}{7}, \frac{3x_0 - 2z_0 + 6t_0}{7} , y_0, \frac{-2x_0 +
      6z_0 + 3t_0}{7}\right)$$
  is another integer solution of the system \eqref{eq:1334}, which we
  can write as:
  \begin{equation}\label{n^2=4Amod7}
    \left\{
      \begin{array}{rcl}
        \hat{x}_1&=&x_0+z_0-\mu\\
        \hat{y}_1&=&-2z_0+3\mu\\
        \hat{z}_1&=&y_0\\
        \hat{t}_1&=&2z_0+t_0-2\mu.
      \end{array}
    \right.
  \end{equation}
\item If $\hat{\gamma}_2\in\ll$, then $n^2\equiv 2A\pmod 7$, and
  $x_0 - y_0 + 2z_0 -t_0 \equiv 0\pmod 7$,
  i.e.~$ x_0 - y_0 + 2z_0 -t_0=7\nu$, for some $\nu\in\ZZ$. As in the
  previous case, one shows that for
  \begin{small}
    $\gamma_2^*= \frac{5x_0 + 2y_0 - 4z_0 + 2t_0}{7} -
    \frac{2x_0 + 5y_0 + 4z_0 - 2t_0}{7}i + \frac{4x_0 - 4y_0 + z_0 -
      4t_0}{7}j - \frac{2x_0 - 2y_0 + 4z_0 + 5t_0}{7}k$
  \end{small}
  we have that
  $\Re\left(\gamma_2^*\beta\right)=\Re(\gamma\beta)$.  Thus, the conjugate of
  $\gamma_2^*$ provides another integer solution of the system
  \eqref{eq:1334}. Denoting its coordinates by
  $\hat{x}_2,\hat{y}_2,\hat{z}_2,\hat{t}_2$, one can see that:
  \begin{equation}\label{n^2=2Amod7}
    \left\{
      \begin{array}{rcl}
        \hat{x}_2&=&x_0-2\nu\\
        \hat{y}_2&=&y_0+2\nu\\
        \hat{z}_2&=&z_0-4\nu\\
        \hat{t}_2&=&t_0+2\nu.
      \end{array}
    \right.
  \end{equation}
\item If $\hat{\gamma}_3\in\ll$, then $n^2\equiv -A\pmod 7$, and
  $x_0 + 4y_0 - 2z_0 \equiv 0\pmod 7$,
  i.e.~$ x_0 + 4y_0 - 2z_0=7\xi $, for some $\xi\in\ZZ$. As in the
  previous cases,
  $$\gamma_3^*=\frac{-2x_0+6y_0-3z_0}{7} +
  \frac{3x_0-2y_0-6z_0}{7}i - \frac{6x_0+3y_0+2z_0}{7}j-t_0k,$$
  satisfies $\Re\left(\gamma_3^*\beta\right)=\Re(\gamma\beta)$.
  Hence, the coordinates of its conjugate,
  $\hat{x}_3,\hat{y}_3,\hat{z}_3,\hat{t}_3$, furnish another integer
  solution of the system \eqref{eq:1334}, and one has:
  \begin{equation}\label{n^2=-Amod7}
    \left\{
      \begin{array}{rcl}
        \hat{x}_3&=&2y_0-z_0-2\xi\\
        \hat{y}_3&=&2y_0-3\xi\\
        \hat{z}_3&=&-3y_0+2z_0+6\xi\\
        \hat{t}_3&=&t_0.
      \end{array}
    \right.
  \end{equation}
\end{itemize}
Now we have everything that we need to prove the following result.
\begin{Prop}\label{135int 3}
  Let $m,n\in\NN$ be such that $35m-n^4$ is non-negative not of the
  form $4^r(8s+7)$, for any $r,s\in\NN$. When $m\equiv -1\pmod 3$ the
  system \eqref{eq:135} has integer solutions for all $n\in\NN$ with $(n,105)=1$.
\end{Prop}

\begin{proof}
  Let $A\in\ZZ$ from \eqref{eq:ABC2}. We see that for all $n\in\NN$
  such that $n^2\not\equiv 0\pmod 7$, we necessarily have one of the
  following: $n^2\equiv \pm A\pmod 7$, $n^2\equiv \pm 2A\pmod 7$,
  $n^2\equiv \pm 4A\pmod 7$, or $A\equiv 0\pmod 7$. Therefore, we
  have:
  \begin{itemize}
  \item If either $A\equiv 0\pmod 7$, $n^2\equiv A\pmod 7$,
    $n^2\equiv -2A\pmod 7$, or $n^2\equiv -4A\pmod 7$, then
    Proposition~\ref{mod7Conditions} says that the system
    \eqref{eq:135} has an integer solution.
  \item If $n^2\equiv -A\pmod 7,$ then the solution
    $\hat{x}_3,\hat{y}_3,\hat{z}_3,\hat{t}_3$ from \eqref{n^2=-Amod7}
    satisfies
    $\hat{x}_3+2\hat{z}_3+2\hat{t}_3\equiv x_0+z_0+2t_0\pmod 3$, and
    $\hat{x}_3+2\hat{y}_3+2\hat{t}_3\equiv x_0+y_0+2t_0\pmod 3$,
    therefore, Proposition~\ref{mod3bad case} and
    Proposition~\ref{mod3Conditions} yield the result.
  \item If $n^2\equiv 4A\pmod 7$, then the solution
    $\hat{x}_1,\hat{y}_1,\hat{z}_1,\hat{t}_1$ from \eqref{n^2=4Amod7}
    satisfies
    $\hat{x}_1+2\hat{y}_1+2\hat{t}_1\equiv 2(x_0+z_0+2t_0)\pmod 3$, and
    $\hat{x}_1+2\hat{z}_1+2\hat{t}_1\equiv 2(x_0+y_0+2t_0)\pmod 3$, so
    again Proposition~\ref{mod3bad case} and
    Proposition~\ref{mod3Conditions} yield the result.
  \item If $n^2\equiv 2A\pmod 7$, then
    $x_0-y_0+2z_0-t_0\equiv 0\pmod 7$, so $ x_0-y_0+2z_0-t_0=7\nu$,
    for some $\nu\in\ZZ$.
    We are going to check when the solution~\eqref{n^2=2Amod7} satisfies
    the solvability conditions modulo $5$ of
    Proposition~\ref{mod5Conditions} and
    Corollary~\ref{mod5ConditionsCor}. Let
    $\hat{A}=3 \hat{x}_2 - \hat{y}_2 - 4 \hat{z}_2 + 3 \hat{t}_2$ be
    the corresponding $A$ for the solution
    $\hat{x}_2,\hat{y}_2,\hat{z}_2,\hat{t}_2$. If either
    $\hat{A}\equiv 0\pmod 5$, $n^2\equiv2\hat{A}\pmod 5$, or
    $n^2\equiv -\hat{A}\pmod 5$ holds, then, by
    Proposition~\ref{mod5Conditions}, we are done. So we just need to
    check the following two cases:
    \begin{itemize}
    \item If $n^2\equiv \hat{A}\pmod 5$, then
      $ \hat{x}_2-2\hat{y}_2-\hat{z}_2+2\hat{t}_2\equiv 0\pmod 5$. 
      Therefore, $ x_0-2y_0-z_0+2t_0\equiv-2\nu\pmod 5$. We also have
      that $x_0-y_0+2z_0-t_0= 7\nu\equiv 2\nu\pmod 5$, and therefore
      we obtain $x_0+y_0-2z_0-2t_0\equiv 0\pmod
      5$. Corollary~\ref{mod5ConditionsCor} then yields the result.
    \item If $n^2\equiv -2\hat{A}\pmod 5$, then
      $\hat{y}_2\equiv 3\hat{x}_2\pmod 5$, which implies that
      $y_0+2x_0\equiv2\nu\pmod 5$. This together with
      $x_0-y_0+2z_0-t_0\equiv 2\nu\pmod 5$ yields      $ x_0+2y_0-2z_0+t_0\equiv 0\pmod 5$. Therefore,
      Corollary~\ref{mod5ConditionsCor} yields the result again.
    \end{itemize}
  \end{itemize}
\end{proof}

Proposition~\ref{135int 1+2} and Propostion~\ref{135int 3} combined
make for Theorem~\ref{Thm135int}.
\section{Integer solutions}

For $m\in\NN$, we set
$S_m=\left\{n\in\NN:\,35m-n^4\geqslant 0\right\}$, and it will also
be convenient to set
$T_m=\left\{n\in S_m:\,35m-n^4 \text{ is a sum of 3 squares}\right\}$.
\begin{Lem}\label{LemOddEven}
  If $m\not\equiv 0\pmod {16}$, then $T_m$ contains either all odd
  numbers of $S_m$, or all even numbers of $S_m.$
\end{Lem}
\begin{proof}
  Simple congruence arguments easily show the following:
  \begin{equation*}
    \begin{array}{lcl}
      m\equiv 1, 3, 7 \pmod{8}
      &\implies& 2\NN\cap S_m \subseteq T_m,\\
      m\equiv 2, 4, 5, 6 \pmod{8}
      &\implies& \left(1+2\NN\right)\cap S_m \subseteq T_m,\\
      m\equiv 8 \pmod{16}  &\implies& 2\NN\cap S_m \subseteq T_m.
      \end{array}
\end{equation*}

\end{proof}

\begin{Lem}
  \label{LemOddEven2}
  If $m\not\equiv 0 \pmod {16} $ and if $A$ is a subset of $S_m$
  containing at least $10$ consecutive numbers, then there is at least
  one $n\in A\cap T_m$ that satisfies $3\mid n$ and $5\nmid n$, and
  another $n\in A\cap T_m$ such that $(n,105)=1$.
\end{Lem}
\begin{proof}
  Consider the classes modulo $15$ written in a circle as
  follows.\bigskip
 
  \begin{center}
    \setlength{\unitlength}{3mm}
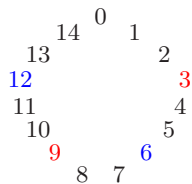
    \begin{picture}(7,7)%
      \footnotesize%
      \put(5,7){0}%
      \put(3.3,6.3){14}%
      \put(6.5,6.3){1}%
      \put(2,5.3){13}%
      \put(7.8,5.3){2}%
      \put(1.2,4.2){\textcolor{blue}{12}}%
      \put(8.7,4.2){\textcolor{red}{3}}%
      \put(1.4,3){11}%
      \put(8.5,3){4}%
      \put(2,2){10}%
      \put(8,2){5}%
      \put(3,1){\textcolor{red}{9}}%
      \put(7,1){\textcolor{blue}{6}}%
      \put(4.2,0){8}%
      \put(5.8,0){7}
    \end{picture}
  \end{center}
  The colors have the following meaning: red and blue represent
  different parities, but not necessarily the parity of the number in
  the figure --- if one adds a multiple of $15$, the parities either
  remain the same or switch ---, and only residues that are divisible
  by $3$ but not by $5$ are colored. Then, in order to apply the
  previous lemma to guarantee that in a certain set of consecutive
  numbers there is at least one divisible by $3$ but not by $5$, one
  just needs to ensure that the set must contain both a blue and a red
  residue. By inspection of the figure, one sees that, actually, one
  only needs $9$ consecutive numbers (the worst cases are the
  sequences starting at $10$ and ending at $3$, and starting at $13$
  and ending at $6$).

  For the second statement, one must work modulo $105$. Again, imagine
  all the classes modulo $105$ in a circular, or periodic
  arrangement. Here we represent them in five lines, and the reader
  should imagine the number $104$ connected back to the beginning, and
  only the residues that are coprime to $105$ are shown, the other
  being represented by a dot.
  $$\scriptsize
  \begin{array}{ccccccccccccccccccccc}
    \cdot & \textcolor{red}{1} & \textcolor{blue}{2} & \cdot &
    \textcolor{blue}{4} & \cdot & \cdot & \cdot & \textcolor{blue}{8}
    & \cdot & \cdot & \textcolor{red}{11} & \cdot &
    \textcolor{red}{13} & \cdot & \cdot & \textcolor{blue}{16} &
    \textcolor{red}{17} & \cdot & \textcolor{red}{19} & \cdot\\  
    \cdot & \textcolor{blue}{22} & \textcolor{red}{23} & \cdot & \cdot
    & \textcolor{blue}{26} & \cdot & \cdot & \textcolor{red}{29} &
    \cdot & \textcolor{red}{31} & \textcolor{blue}{32} & \cdot
    & \textcolor{blue}{34} & \cdot & \cdot & \textcolor{red}{37} &
    \textcolor{blue}{38} & \cdot & \cdot & \textcolor{red}{41}\\
    \cdot & \textcolor{red}{43} & \textcolor{blue}{44} & \cdot &
    \textcolor{blue}{46} & \textcolor{red}{47} & \cdot & \cdot &
    \cdot & \cdot & \textcolor{blue}{52} & \textcolor{red}{53} &
    \cdot & \cdot & \cdot & \cdot & \textcolor{blue}{58} &
    \textcolor{red}{59} & \cdot & \textcolor{red}{61} &
    \textcolor{blue}{62} \\
    \cdot & \textcolor{blue}{64} & \cdot & \cdot &
    \textcolor{red}{67} & \textcolor{blue}{68} & \cdot & \cdot &
    \textcolor{red}{71} & \cdot & \textcolor{red}{73} &
    \textcolor{blue}{74} & \cdot & \textcolor{blue}{76} &
    \cdot & \cdot & \textcolor{red}{79} & \cdot & \cdot &
    \textcolor{blue}{82} & \textcolor{red}{83}\\
    \cdot & \cdot & \textcolor{blue}{86} & \cdot &
    \textcolor{blue}{88} & \textcolor{red}{89} & \cdot & \cdot &
    \textcolor{blue}{92} & \cdot & \textcolor{blue}{94} & \cdot &
    \cdot & \textcolor{red}{97} & \cdot & \cdot & \cdot &
    \textcolor{red}{101} & \cdot & \textcolor{red}{103} &
    \textcolor{blue}{104}
  \end{array}
  $$
  Again, a simple inspection shows that $10$ consecutive numbers
  suffice to guarantee at least a blue and a red residue (the worst
  cases are the sequences starting at $2$ and ending at $11$, and
  starting at $95$ and ending with $104$).
\end{proof}

\bigskip

We have that $|S_m|\geqslant 10$ if $\sqrt[4]{35m}\geqslant 10$, which
is equivalent to $ m\geqslant 286$. Therefore, from
Theorem~\ref{Thm135int}, it follows that the system \eqref{eq:135} has
integer solutions for all $m\not\equiv0 \pmod{16} $ and
$m\geqslant 286.$ Since it is easy to check that this system has
solutions for all $m$ up to $286$, and since a solution
$(x_0,y_0,z_0,t_0)\in\ZZ^4$ for that system for some $m$, yields the
solution $(4 x_0,4 y_0,4 z_0,4 t_0)\in\ZZ^4$ for $16m$, a simple
descent argument establishes the following result.
\begin{Thm}
  \label{135int}
  Any $m \in \NN$ can be written as $x^2 + y^2 + z^2 + t^2$ with
  $x, y, z, t \in \ZZ$ such that $x + 3y + 5z$ is a square. Moreover,
  for $m\in\NN$, with $16\nmid m$ one can choose this square to be one
  of $1,4,9$, or $36$.
\end{Thm}

\begin{proof}
  It only remains to prove the last statement, which follows from the
  fact that, when $6\in S_m$, i.e.~when $m\geq 38$, by
  Lemma~\ref{LemOddEven}, $T_m$ either contains $\{1,3,5\}$ or
  $\{2,4,6\}$. Thus, for $m\equiv 0, -1\pmod{3}$ one can choose, in
  Theorem~\ref{Thm135int}, either $n=1$ or $n=2$; for
  $m\equiv 1\pmod{3}$, either $n=3$ or $n=6$.
\end{proof}
Note that if we do not require $16\nmid m$ then the square will be on
the set
$$\lbrace 4^rs^2 : r\in2\NN \ \& \ s\in\lbrace1,2,3,6\rbrace\rbrace.$$
In Conjecture 4.5(ii) of the paper \cite{sun2019restricted}, Zhi-Wei
Sun conjectured that any $n\in\NN$ can be written as
$x^2 + y^2 + z^2 + t^2$ with $x, y, z, t \in \NN$ such that
$\mid x + 3y - 5z\mid\in\lbrace 4^r: r\in\NN \rbrace$. Theorem
\ref{135int} provides an advance towards this conjecture.
\section{Natural Solutions}

\begin{Thm}\label{Nat135}
  For $m\in\ZZ$ not divisible by $16$ and sufficiently large
  $(\text{namely } m\geqslant 1.05104\times10^{11})$, there exists at
  least one
  $n\in\displaystyle\left[\sqrt[4]{34\, m},\sqrt[4]{35\, m}\right]$
  such that the system \eqref{eq:135} has solutions in $\NN$.
\end{Thm}

\begin{proof}
  Firstly, we note that there is a constant $c\in\RR$ such that for
  $m\geq c$ the interval
  $\displaystyle\left[\sqrt[4]{34\, m},\sqrt[4]{35\, m}\right]$
  contains at least $10$ consecutive integers. We can easily calculate
  $c$:
  $$\sqrt[4]{35\, m}-\sqrt[4]{34\, m}\geq 10\iff
  m\geq\left(\frac{10}{\sqrt[4]{35}-\sqrt[4]{34}}\right)^4\simeq
  105\,103\,560\,126.8026.$$
  From Lemma \ref{LemOddEven2} we know that, for $m\geqslant c$, the
  interval contains an $n\in\NN $ such that $35m-n^4$ is a sum of 3
  squares and $m,n$ satisfy the conditions of Theorem
  \ref{Thm135int}. It then follows that there exist $A,B,C\in\ZZ$ such
  that $\delta=\gamma\alpha=n^2+Ai+Bj+Ck\in\ll$, for some
  $\gamma=x-yi-zj-tk\in\ll$, with $\alpha=1+3i+5j$, and
  $\Norm(\delta)=35\,m$.  We then have
  that $$\delta=(x+3y+5z)+(3x-y+5t)i+(5x-z-3t)j+(5y+3z-t)k.$$
  Therefore we must have that
  \begin{equation}\label{eq:ABC}
    \left\{
      \begin{array}{rcl}
        n^2 &=&   x + 3 y + 5 z\\
        A &=&   3 x  -  y + 5 t\\
        B &=&   5 x  -  z - 3 t\\
        C &=& - 5 y + 3 z  -  t.
      \end{array}
    \right.
  \end{equation}
  Solving \eqref{eq:ABC} yields
  \begin{equation*}
    \left\{
      \begin{array}{rcl}
        x &=& \frac{3A+5B+n^2}{35}\\
        y &=& \frac{-A-5C+3n^2}{35}\\
        z &=& \frac{-B+3C+5n^2}{35}\\
        t &=& \frac{5A-3B-C}{35}.
      \end{array}
    \right.
  \end{equation*}
  Note that if $(x,y,z,t)$ is a solution of \eqref{eq:135}, then
  $(x,y,z,-t)$ is a solution of it as well.  Therefore a sufficient
  condition to have a solution of \eqref{eq:135} in $\NN$ is:
  \begin{equation}\label{ineq135}
    \left\{
      \begin{array}{rcl}
       n^2 &\geqslant& -3A-5B\\
        3n^2 &\geqslant& A+5C\\
        5n^2 &\geqslant& B-3C.\\
      \end{array}
    \right.
  \end{equation}
  Now, from the Cauchy-Schwartz inequality we can see that
  \begin{equation*}
    (3|A|+5|B|)^2 \leqslant (3^2+5^2)(A^2+B^2)\leqslant
    34\,(A^2+B^2+C^2) = 34\,(35m-n^4). 
  \end{equation*}
  If $n^2\geqslant \sqrt{34\,(35m-n^4)},$ then
  $n^2\geqslant-3A-5B,$ and so $ x\geqslant 0.$ Similarly one can
  show that if $3n^2\geqslant \sqrt{26\,(35m-n^4)},$ then
  $ y\geqslant 0,$ and if $5n^2\geqslant \sqrt{10\,(35m-n^4)},$ then
  $ z\geqslant 0.$ Hence $n^2\geqslant \sqrt{34\,(35m-n^4)}$ is a
  sufficient condition for $x,y,z\in\NN.$ The last condition is
  equivalent to $n\geqslant\displaystyle\sqrt[4]{34\,m}$.
\end{proof}
 
%

Finally, Rog\'erio Reis, from the department of Computer Science of
the University of Porto, and a researcher at CMUP, wrote a very
efficient C program, implementing ideas by the authors and suggestions
made by Zhi-Wei Sun, that checked that all natural numbers up to
$105\,103\,560\,126$, except for the multiples of 16, do have a 1-3-5
representation. These verification is reported in
\cite{report2020}. Qing-Hu Hou computation mentioned in
\cite{WuSun2020} was also rechecked, i.e.~it was verified that Sun's
1-3-5 Conjecture holds for all numbers up to $10^{10}$. Since
$\frac{105\,103\,560\,126}{16} < 10^{10}$, a simple descent argument
completes the proof of the 1-3-5 conjecture. Therefore, we can now
state the following.
 
\begin{Mthm}
Any natural number can be written as a sum of four squares,
$x^2+y^2+z^2+t^2$ with $x, y, z, t\in\NN$, in such a way that
$x+3y+5z$ is a perfect square.
\end{Mthm}

As a final remark, we note that what one would naturally call the
1-3-3-4 conjecture is not true. That is, it is not true that every
natural number $m$ can be written as a sum of four squares,
$m=x^2+y^2+z^2+t^2$, so that $x+3y+3z+4t$ is a perfect square. For
example the numbers 3, 4, 7, 8, 22, 23, 31, 42, 61, 95, 148, 157 and
376 do not have such a representation.  Computations seem to suggest
that, except for these thirteen numbers and all its multiples by
powers of $16$, all other numbers do have a 1-3-3-4 representation.
\section*{Acknowledgments}
  
The authors would like to thank Professor Zhi-Wei Sun
(\hanyu{孙智伟}) and Professor Bo He (\hanyu{何波}) for their
kind feedback and helpful comments to lower the bound for the natural
solutions, and for suggestions to design an effective algorithm to
finish the verification of the 1-3-5 conjecture.

The authors would also like to acknowledge the financial support by
FCT --- Funda\c{c}\~ao para a Ci\^encia e a Tecnologia, I.P.---,
through the grants for Nikolaos Tsopanidis with references
PD/BI/143152/2019, PD/BI/135365/2017, PD/BI/113680/2015, and by CMUP
--- Centro de Matem\'atica da Universidade do Porto ---, which is
financed by national funds through FCT under the project with
reference UID/MAT/00144/2020.

\begin{thebibliography}{99}
\bibitem{ankeny1957sums} NC Ankeny. Sums of Three
  Squares.\textit{ Proc.~Am.~Math.~Soc.}, 8(2):316–319, 1957.
\bibitem{conway2003quaternions} John H Conway and Derek A
  Smith. \textit{On Quaternions and Octonions}. AK Peters/CRC Press,
  2003.
\bibitem{legendre1808essai} Adrien-Marie Legendre. \textit{Essai sur
    la Th\'eorie des Nombres}. Courcier, 1808.
\bibitem{lehmer1948partition} Derrick H Lehmer. On the Partition of
  Numbers into Squares. \textit{Amer.~Math.~Monthly}, 55(8):476–481,
  1948.
\bibitem{report2020} Ant\'onio Machiavelo, Rog\'erio Reis, and
  Nikolaos Tsopanidis. Report on Zhi-Wei Sun’s 1-3-5 Conjecture and
  Some of Its Refinements. \textit{Journal of Number Theory},
  222:21–29, 2021.
\bibitem{pall1940arithmetic} Gordon Pall. On the Arithmetic of
  Quaternions. \textit{Trans.~Am.~Math.~Soc.}, 47(3):487– 500, 1940.
\bibitem{sun2018variants} Yu-Chen Sun and Zhi-Wei Sun. Some Variants
  of Lagrange’s Four Squares Theorem.  \textit{Acta Arith.},
  183:339–356, 2018.
\bibitem{sun2017refining} Zhi-Wei Sun. Refining Lagrange’s Four-Square
  Theorem. \textit{Journal of Number Theory}, 175:167– 190, 2017.
\bibitem{sun2019restricted} Zhi-Wei Sun. Restricted Sums of
  Four-Squares. \textit{Int.~J.~Number Theory}, 15(9):1863– 1893,
  2019.
\bibitem{wojcik1971sums} Jan W\'ojcik. On Sums of Three
  Squares. \textit{Colloq.~Math.}, 1(24):117–119, 1971.
\bibitem{WuSun2020} Hai-Liang Wu and Zhi-Wei Sun. On the 1-3-5
  Conjecture and Related Topics. \textit{Acta Arith.}, 193:253–268,
  2020.
\end{thebibliography}

\begin{thebibliography}{99}
\bibitem{JorgeMSc} Jorge Costa.  Quaternions, Arithmetic and
  Diophantine Equations. Master's Thesis, Faculdade de Ci\^encias da
  Universidade do Porto (2023). Available at
  \url{https://hdl.handle.net/10216/156535}
\bibitem{135ConjVariations} Ant\'onio Machiavelo and Nikolaos
  Tsopanidis. Zhi-Wei Sun's 1-3-5 Conjecture and
  Variations. \textit{Journal of Number Theory}, 222:1--20, 2021.
\bibitem{135Report} Ant\'onio Machiavelo, Rog\'erio Reis and Nikolaos
  Tsopanidis. Report on Zhi-Wei Sun's 1-3-5 Conjecture and Some of Its
  Refinements. \textit{Journal of Number Theory}, 222:21--29, 2021.
\end{thebibliography}
%


\vspace{2mm}

\hrule

\vspace{2mm}

\begin{center}
  \large Corrigendum to\\[2mm]
  ``Zhi-Wei Sun's 1-3-5 Conjecture and Variations''\\[2mm]
  \normalsize [Journal of Number Theory 222 (2021) 1--20]\\[4mm]
  Jorge Costa, Ant\'onio Machiavelo\,\orcidlink{0000-0002-7595-7275},
  Rog\'erio Reis\,\orcidlink{0000-0001-9668-0917}
\end{center}

\begin{abstract}
  There is an error is the proof of Lemma~14 of
  \cite{135ConjVariations} that, after correction, implies a larger
  bound beyond which the 1-3-5 conjecture is guarantee to hold without
  explicy computations. We correct that error here, and report on the
  computer verification we made to ensure that the conjecture is
  indeed true.
\end{abstract}

\maketitle

\footnotesize

\textbf{Keywords:} Quaternions, Lipschitz integers, 1-3-5
conjecture.

\normalsize

\bigskip
While working on his Master Thesis \cite{JorgeMSc}, the first author
noticed that there is a mistake in the proof of Lemma~14 of
\cite[p.~16]{135ConjVariations}. Namely, when considering the circular
arrangement of the classes modulo $15$ at the beginning of that proof,
the fact that there is a switch of the colors when one goes once
around the circle was not properly taken into account in the deduction
of the minimal number of consecutive integers needed to guarantee that
at least one is divisible by $3$ but not by $5$. It was said that $9$
consecutive numbers are enough, while in fact one needs $12$, as can
easily be seen when considering the classes modulo $30$, as in the
following figure (the colors having the same meaning as in
\cite{135ConjVariations}):
$$
\begin{array}{ccccccccccccccc}
  0&1&2&\textcolor{red}{3}&4&5&\textcolor{blue}{6}&7&8&\textcolor{red}{9}&10&11&\textcolor{blue}{12}&13&14\\
  15&16&17&\textcolor{blue}{18}&19&20&\textcolor{red}{21}&22&23&\textcolor{blue}{24}&25&26&\textcolor{red}{27}&28&29
\end{array}
$$

This implies that the third line of the proof of Theorem~16 in
\cite[p.~16]{135ConjVariations} should be changed to:
$$\sqrt[4]{35\, m}-\sqrt[4]{34\, m}\geq 12\iff
m\geq\left(\frac{12}{\sqrt[4]{35}-\sqrt[4]{34}}\right)^4\simeq
217\,942\,742\,278.94$$
and Theorem~16, on p.~17, should now read:
\begin{Thm}
  For $m\in\ZZ$ not divisible by $16$ and sufficiently large, namely
  for $m\geqslant 217\,942\,742\,279$, there exists at least one
  $n\in\displaystyle\left[\sqrt[4]{34\, m},\sqrt[4]{35\, m}\right]$
  such that the system 1-3-5 has solutions in $\NN$.
\end{Thm}

We have used the program described in \cite{135Report} to check all
the extra numbers up to the new limit, which took about 10 days of CPU
time in a ARM Apple M2, and can hereby confirm that the 1-3-5
conjecture is indeed true.

\section*{Acknowledgments}

The authors were partially supported by CMUP, member of LASI, which is
financed by national funds through FCT --- Funda\c{c}\~ao para a
Ci\^encia e a Tecnologia, I.P.--- , under the projects with reference
UIDB/00144/2020 and UIDP/00144/2020.

\end{document}